\documentclass[AMA,Times1COL]{WileyNJDv5} 

\articletype{Research article}%

\received{Date Month Year}
\revised{Date Month Year}
\accepted{Date Month Year}
\journal{Journal}
\volume{00}
\copyyear{2023}

\startpage{1}

\setcitestyle{numbers,open={[},close={]}}

\renewenvironment{flushright}{%
    \begin{center}%
}{%
    \end{center}%
    \ignorespacesafterend%
}

\usepackage{graphicx}      
\graphicspath{{./tikz/}}
\usepackage{xcolor}
\usepackage[utf8]{inputenc}

\usepackage{algorithm}
\usepackage{algpseudocode}

\usepackage{subfig}

\usepackage{array}   
\newcolumntype{L}{>{$}l<{$}} 

\usepackage{svg}
\usepackage{tikz}
\usepackage{booktabs}
\usepackage{pgfplots}
\usepackage{tikzscale}
\usepackage{pgfkeys}
\pgfplotsset{compat=newest}

\usetikzlibrary{arrows.meta}
\usetikzlibrary{backgrounds}
\usepgfplotslibrary{patchplots}
\usepgfplotslibrary{fillbetween}
\pgfplotsset{%
layers/standard/.define layer set={%
    background,axis background,axis grid,axis ticks,axis lines,axis tick labels,pre main,main,axis descriptions,axis foreground%
}{grid style= {/pgfplots/on layer=axis grid},%
    tick style= {/pgfplots/on layer=axis ticks},%
    axis line style= {/pgfplots/on layer=axis lines},%
    label style= {/pgfplots/on layer=axis descriptions},%
    legend style= {/pgfplots/on layer=axis descriptions},%
    title style= {/pgfplots/on layer=axis descriptions},%
    colorbar style= {/pgfplots/on layer=axis descriptions},%
    ticklabel style= {/pgfplots/on layer=axis tick labels},%
    axis background@ style={/pgfplots/on layer=axis background},%
    3d box foreground style={/pgfplots/on layer=axis foreground},%
    },
}
\pgfplotsset{ /tikz/every picture/.append style={trim axis left, trim axis right}}

\usepgfplotslibrary{external} 
\tikzset{external/system call={pdflatex \tikzexternalcheckshellescape 
                                        -halt-on-error
                                        -interaction=batchmode 
                                        -jobname "\image" "\texsource"
                                        && pdftops -eps "\image.pdf"}}

\usepackage{amsmath,amssymb,amsfonts}
\usepackage{textcomp}
\usepackage{printlen}
\usepackage{mathtools}
\usepackage{placeins}
\usepackage{siunitx}
\usepackage{bm}
\usepackage{nicefrac}
\newcommand{\bs}[1]{\boldsymbol{{#1}}}
\newcommand{\vect}[1]{\boldsymbol{\mathbf{#1}}}
\usepackage{xcolor}

\definecolor{dkred}{rgb}{0.0,0,0}
\definecolor{dkgreen}{rgb}{0,0,0}
\definecolor{corrige}{rgb}{0,0,0}
\definecolor{corrige2}{rgb}{0,0,0}
\definecolor{corrige3}{rgb}{0,0,0.0}
\definecolor{color_MKres}{rgb}{0,0.0,0.0}
\definecolor{corrigeplane}{rgb}{0,0,0}
\definecolor{vclast}{rgb}{0,0.0,0}
\definecolor{mkcor}{rgb}{0.0,0.0,0}
\definecolor{corrcol}{rgb}{0.0,0.0,0.0}

\newcommand{\corr}[1]{{\color{corrcol}{#1}}}

\makeatletter
\newcommand\deflabel[1]{\def\@currentlabel{#1}}

\newcommand{\pushright}[1]{\ifmeasuring@#1\else\omit\hfill$\displaystyle#1$\fi\ignorespaces}
\makeatother

\usepackage{algorithm}
\usepackage{amsmath}

\DeclareMathOperator*{\argmin}{arg\!\min}
\DeclareMathOperator*{\D}{d\!}

\usepackage{hyperref}

\makeatletter
\newcommand{\bigcomp}{%
  \DOTSB
  \mathop{\vphantom{\sum}\mathpalette\bigcomp@\relax}%
  \slimits@
}
\newcommand{\bigcomp@}[2]{%
  \begingroup\m@th
  \sbox\z@{$#1\sum$}%
  \setlength{\unitlength}{0.9\dimexpr\ht\z@+\dp\z@}%
  \vcenter{\hbox{%
    \begin{picture}(1,1)
    \bigcomp@linethickness{#1}
    \put(0.5,0.5){\circle{1}}
    \end{picture}%
  }}%
  \endgroup
}
\newcommand{\bigcomp@linethickness}[1]{%
  \linethickness{%
      \ifx#1\displaystyle 2\fontdimen8\textfont\else
      \ifx#1\textstyle 1.65\fontdimen8\textfont\else
      \ifx#1\scriptstyle 1.65\fontdimen8\scriptfont\else
      1.65\fontdimen8\scriptscriptfont\fi\fi\fi 3
  }%
}

\newcommand{\new}[1]{{\color{black}#1}}

\makeatother

\newcommand{\inputnamedtex}[1]{%
    \tikzsetnextfilename{#1}%
	    \includegraphics{tikz/#1.eps}
}

\begin{document}

\title{Towards Optimal Spatio-Temporal Decomposition of Control-Related Sum-of-Squares Programs
}

\author[1,2]{Vít Cibulka}

\author[1,2]{Milan Korda}

\author[1]{Tomáš Haniš}

\authormark{CIBULKA \textsc{et al.}}
\titlemark{Towards Optimal Spatio-Temporal Decomposition of Control-Related Sum-of-Squares Programs
}


\address[1]{\orgdiv{Department of Control Engineering}, \orgname{Faculty of Electrical Engineering, Czech Technical University in Prague}, \country{The Czech Republic}}

\address[2]{\orgdiv{CNRS}, \orgname{Laboratory for Analysis and Architecture of Systems}, \orgaddress{\state{Toulouse}, \country{France}}}

\corres{Corresponding author Vít Cibulka,  \email{cibulka.vitek@gmail.com}}


\fundingInfo{Czech Science Foundation (GACR)
under contracts No. GA19-18424S, GA20-11626Y, 
and by the Grant Agency of the Czech Technical University in Prague,
grant No. SGS19/174/OHK3/3T/13. This work was also co-funded by the European Union under the project ROBOPROX (reg.~no.~CZ.02.01.01/00/22\_008/0004590) as well as was part of the programme DesCartes supported by the National Research Foundation, Prime Minister's Office, Singapore under its Campus for Research Excellence and Technological Enterprise (CREATE) programme.}


\abstract[Abstract]{
This paper presents a method for calculating the Region of Attraction (ROA) of nonlinear dynamical systems, both with and without control. The ROA is determined by solving a hierarchy of semidefinite programs (SDPs) defined on a splitting of the time and state space. Previous works demonstrated that this splitting could significantly enhance approximation accuracy, although the improvement was highly dependent on the ad-hoc selection of split locations. In this work, we eliminate the need for this ad-hoc selection by introducing an optimization-based method that performs the splits through conic differentiation of the underlying semidefinite programming problem. We provide the differentiability conditions for the split ROA problem, prove the absence of a duality gap, and demonstrate the effectiveness of our method through numerical examples.

}



\maketitle

\renewcommand\thefootnote{}
\footnotetext{\textbf{Abbreviations:} ROA, region of attraction; SDP, semi-definite program; LP, linear program; SOS, sum of squares.}

\renewcommand\thefootnote{\fnsymbol{footnote}}
\setcounter{footnote}{1}

\section{Introduction}

This work addresses stability and reachability analysis on controlled nonlinear dynamical systems. The method used in this paper assesses these properties by calculating the Region of Attraction (ROA) of a given target set. The reachable set is obtained by reversing the time.

The majority of the currently used methods deal with autonomous systems, where the approximation of ROA is obtained from level sets of Lyapunov functions \new{(see, e.g., \cite{chesi2011domain})}. In the case of polynomial systems, one can obtain the Lyapunov functions by solving a series of semidefinite programs (SDPs), similar to our method. However, these approaches do not consider control systems and are limited to the equilibria of the considered systems. 

This paper uses the works from \cite{milanroa} and \cite{Cibulka2022}, which provide converging outer approximations of the ROA for controlled polynomial systems with respect to a given target set, not necessarily an equilibrium. The original version of this method, introduced in \cite{milanroa}, calculates the ROA by transforming the problem into an infinite-dimensional Linear Program (LP) on measures, whose dual is then relaxed into a problem on sum-of-squares polynomials of a given degree which can be equivalently written as an SDP.

The unsolved challenge of that approach was the rapid growth rate of the size of the SDP with respect to the degree of the used polynomials.
The work \cite{Cibulka2022} provided a remedy for this issue in the form of spatio-temporal decomposition of the problem and its variables. The new problem was obtained by splitting the time and state space into several smaller subsets resulting in multiple smaller, interconnected SDPs of lower complexity. 
The work provided significant improvement in terms of time and memory demands, but left one question unanswered: How to split the time and state space?

This work is a step toward answering this question by proposing an optimization-based method for finding the split configuration. We use the recent work on conic differentiation \cite{diff_sdp2} to obtain a gradient of the related sum-of-squares problem with respect to the parametrization of the splits (e.g., the split positions). The parameters are then optimized by a first-order method. We provide the 
conditions for differentiability of the split problem and show the effectiveness of the method on the examples used in \cite{Cibulka2022}.

\new{This work can be seen as complementary to sparsity-based approaches for complexity reduction of the SDP-based methods~\cite{wang2023exploiting,tacchi2020approximating,schlosser2020sparse} and could be combined with them. The same philosophy of splitting the state space can be applied to other problem classes amenable to these methods such as the invariant set~\cite{korda2014convex} or invariant measure computation~\cite{korda2021convex}, extreme values~\cite{fantuzzi2020bounding} or partial differential equations~\cite{korda2022moments}.}

\textbf{Structure of the paper}
 Section \ref{sec:problem_statement}  defines the Region of Attraction (ROA) and motivates this work. Section \ref{sec:splitting} presents the split ROA problem and the conditions for strong duality. The differentiation is described in Section \ref{sec:sdp_diff} along with the proof of differentiability. The numerical results are in Section \ref{sec:numerical_examples} and we conclude in Section \ref{sec:conclusion}.

\textbf{Notation}
The set of consecutive integers from 1 to \(n\) is denoted by \(\mathbb{Z}_{n}\). The Lebesgue measure (the volume) of a set \(A\) is denoted by \(\lambda(A)\).

\section{Problem Statement}
\label{sec:problem_statement}
We shall focus on the problem of calculating the region of attraction (ROA) of a controlled nonlinear dynamical system
\begin{equation}
	\label{eq:nlsys}
	\dot{x}(t) = f(t,x(t),u(t)),\;\; t \in [0,T],
\end{equation}
where \(x(t) \in \mathbb{R}^n\) is the state vector, \(u(t) \in \mathbb{R}^m\) is the control input vector, \(t\) is time, \(T > 0\) is the final time 
and \(f\) is the vector field, which is assumed to be polynomial in variables \(x\) and \(u\). The state and control inputs are constrained to lie in \mbox{basic semialgebraic sets}
\begin{equation}
	\label{eq:semialg_sets}
\begin{alignedat}{2}
 	&u(t) \in U   :=&&\{u \in \mathbb{R}^m : g_j^U(u)       \geq 0, j \in \mathbb{Z}_{n_U}\}, t\in [0,T] ,\\
 	&x(t) \in X   :=&&\{x \in \mathbb{R}^n : g_j^X(x)       \geq 0, j \in \mathbb{Z}_{n_X}\}, t\in [0,T] ,\\
 	&x(T) \in X_T :=&&\{x \in \mathbb{R}^n : g_j^{ X_T }(x) \geq 0, j \in \mathbb{Z}_{N_{X_T}}\},
\end{alignedat}
\end{equation}
where \(g_j^U(u)\),  \(g_j^X(x) \), and \(g_j^{X_T}(x)\) are polynomials. The Region of Attraction (ROA) is then defined as 
\begin{equation}
	\label{eq:ROA_def}
\begin{alignedat}{2}
	X_0 = \{& x_0 \in &&X : \exists\, u(\cdot) \in L([0,T];U) \\
	&\text{s.t.}&& \dot{x} = f(t,x(t),u(t)) \text{ a.e. on}\;  [0,T],\\
	& &&x(0) = x_0,\;x(t) \in X\;\forall\, t\in [0,T],\; x(T) \in X_T  \},
\end{alignedat}
\end{equation}
where ``a.e.'' stands for ``almost everywhere'' with respect to the Lebesgue measure and $L([0,T];U)$ denotes the space of measurable functions from $[0,T]$ to $U$.

The work \cite{milanroa} introduced an algorithm for calculating an outer approximation of the ROA, based on relaxing a polynomial optimization problem 
into a semidefinite program based on the sum-of-squares (SOS) techniques. In \cite{Cibulka2022}, we improved the accuracy of the SOS relaxation by discretizing (or splitting) the time and state space thus allowing for tighter approximations with favorable scaling properties. It was, however, unclear how to perform this splitting to obtain optimal results, which is the focus of this paper.

We build upon the algorithm from \cite{Cibulka2022}, which increased the accuracy by splitting the sets \(X\) and \([0,T]\) into multiple subsets,
and consider the positions of splits as parameters that are then optimized using gradient descent. The following example shows the advantage of splitting the problem and then optimizing the splits.

\subsection*{Motivation example}
Let us motivate the idea of optimizing the split positions.
Consider a simple double integrator system
\begin{align}
\begin{split}
	\label{eq:double_integrator_motivation}
 \dot{x}_1 &= x_2 \\
\dot{x}_2 &= u
\end{split}
\end{align} 
with \(X = [-0.7,0.7] \times [-1.2,1.2]\) and \(U = [-1,1]\).
 Figure \ref{fig:motiv_roa} shows the difference between the original approach from \cite{milanroa} (salmon), the improved version from \cite{Cibulka2022} (purple), and the version provided in this paper (blue). The real ROA is depicted in green. We can see that each version provides a significant improvement in accuracy. The most notable detail about the ROA calculated by the proposed algorithm (blue) is that it has exactly the same memory demands as the ROA calculated via the method \cite{Cibulka2022} (purple), which provides substantially worse estimation.
\begin{figure}[!htb]
	\centering
		\includegraphics[width = 0.5\textwidth]{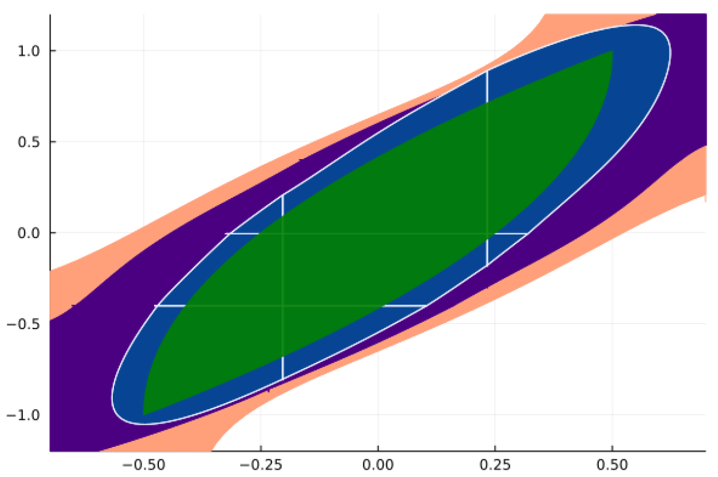}
\caption{ROA approximations of the double integrator.
The salmon approximation \new{was obtained with degree 4 polynomials} and no splits \cite{milanroa}, the purple approximation with 
\new{degree 4 polynomials} and 4 equidistantly-placed splits \cite{Cibulka2022}, and the blue \new{approximation with degree 4 polynomials and 4 splits} with optimized positions. The green set is the ground truth.
The memory demands for blue and purple are exactly the same.
}
\label{fig:motiv_roa}
\end{figure}


\section{Discretizing time and state space}
\label{sec:splitting}
Although this work is mainly concerned with the SDP formulation of the ROA problem, we will first introduce the polynomial problem in order to give meaning to the parameters, which will be optimized through the SDP formulation.

We discretize the time axis as 
\begin{equation}
	[0,T] = \bigcup\limits_{k=1}^{K-1} [T_k,T_{k+1}], \text{ where } T_1 = 0, T_K = T
\end{equation}
and the state space as
\begin{equation}
	X = \bigcup\limits_{i=1}^{I} X_{i},
\end{equation}
where  \(K \geq 2\) is the number of time splits \corr{(including the boundaries)} and \(I\) is the number 
of closed subsets \(X_i\).
In this work, we assume that the subsets \(X_i\) are created by splitting the state space into axis-aligned boxes. Therefore, the \new{interiors of the} sets \(X_i\) do not intersect but \new{the sets} can share a zero-volume boundary. \corr{Therefore the boundary between two neighboring cells $X_a$ and $X_b$ is $X_a \cap X_b$. }
The time-splits are the scalars \(T_k\), excluding the fixed boundary terms; we denote this set as 
\begin{equation}
	\theta_{T} = \{ \corr{T_{2}}, \dots, T_{K-1} \}.
\end{equation}
We define the set of all neighboring subsets \(X_i\) as
\begin{equation}
	N_X := \{(a,b): X_a \cap X_b \not= \emptyset\}.
\end{equation}
The set of state-space parameters is denoted as \(\theta_{X}\) and depends on the chosen parametrization of the splits. In this work, the state space is divided by hyperplanes which are formed by splitting the axes into intervals; this restricts the \(X_i\)'s to axis-aligned boxes as mentioned earlier. The locations of the splits are the parameters in \(\theta_{X}\). An illustration is provided in the Figure \ref{fig:ex_splits}. 
\begin{figure}[!htb]
	\centering
		\includegraphics{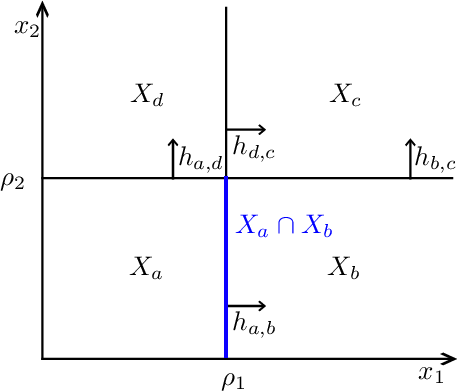}
\caption{Illustration of the split sets \(X_i\), the set boundaries \(x_{a,b}\), their normal vectors \(h_{a,b}\),
and the parameter set \(\theta_X\), for which we get \(\theta_X = 
\{\rho_1,\rho_2\}\) in this example. 
}
\label{fig:ex_splits}
\end{figure}

We shall now describe the steps required to turn the problem of finding the ROA into a semidefinite program (SDP). We will describe the polynomial formulation, then relax it via SOS, and then reformulate it as an SDP.

\subsection{Polynomial formulation}
\label{sec:poly_formulation}
The outer approximation of the ROA \(X_0\)
can be obtained as a super-level set \(	\bar{X}_{0} := \{x : v(0,x) \geq 0\}\) of a piece-wise polynomial
\begin{equation}
	\label{eq:split_v}
	v(t,x) = 
	\begin{cases}
		v_{1,1}(t_1,x_1) &\text{for } t \in [T_1,T_{2}], x \in X_1\\
		\dots\\
		v_{i,k}(t_k,x_i) &\text{for } t \in [T_k,T_{k+1}], x \in X_i\\
		\dots\\
		v_{I,K-1}(t_{K-1},x_I) &\text{for } t \in [T_{K-1},T_{K}], x \in X_I
	\end{cases}
\end{equation}
which is a solution to the problem
\begin{equation}
	\label{eq:dualOriginalSPlit}
	\begin{split}
	\begin{alignedat}{3}
		& d_\mathrm{s}^\star = &&\inf \sum_i \int_{X_i} w_i(x)d\lambda(x)     \\
& \text{s.t.\,}&& \text{for all } i \in \mathbb{Z}_{I}, k \in \mathbb{Z}_{K-1} \text{ and }\\
& && \hspace{3cm} (a,b) \in N_X, && \   x_{a,b}\in X_a\cap X_b\\
		&&& ( \mathcal{L} v_{i,k} )(t,x,u) \leq 0  \quad\,  \forall(t,x,u)  \in [&&T_k,T_{k+1}] \times X_i \times U    \\
	&	&&w_i(x)                \geq v_{i,1}(0,x) + 1  \quad &&\forall x \in X_i   \\
&		&&v_{i,K-1}(T,x)              \geq 0   \quad &&\forall x \in X_T \\
&    &&w_i(x)                \geq 0         \quad &&\forall x \in X_i \\
&	&& v_{i,k}(T_{k+1},x) \geq v_{i,k+1}(T_{k+1},x) \quad &&\forall x \in X_i\\
&	&&( v_{a,k}(t,x_{a,b}) - v_{b,k}(t,x_{a,b}) )\cdot  h_{a,b}^\top&&f(t,x_{a,b},u)  \geq 0,
	\end{alignedat}\\
\end{split}
\end{equation}
where the polynomials \(w_i\) and \(v_{i,k}\) are the problem variables and
 \(h_{a,b}\) is the normal vector of the boundary $X_a\cap X_b$,
 \corr{which is assumed to be element-wise positive.}  
For more insight, we refer the reader to \cite{Cibulka2022}.

The problem is parametrized by the time splits \(\theta_T\) and the state splits \(\theta_X\) (which define 
the boundaries \(X_a\cap X_b\)).
Let us define the vector of the parameters as 
\begin{equation}
	\theta = \theta_T \cup \theta_X.
\end{equation}

\subsection{Sum-of-squares relaxation}
\label{sec:sos_relax}
The SOS approximation  of \eqref{eq:dualOriginalSPlit} reads
\begin{equation}
	\label{eq:dualSOS}
	\begin{alignedat}{3}
		& &&  \inf \sum_i \text{ }  \mathbf{w}_i^\top l_i\\
&\text{s.t.\quad} &&\text{for all } i \in \mathbb{Z}_{I}, k \in \mathbb{Z}_{K-1} \text{ and }  (a,b) \in N_X \\
		& &&-(\mathcal{L}v_{i,k})(z) =  q_{i,k}(z) + \bs{s^\tau_{i,k}}(z)^\top\bs{g^{\tau}_k}(t)\\
		& &&\hspace{2.2cm} + \bs{s^X_{i,k}}(z)^\top \bs{g^X_{i,k}}(x) + \bs{s^U_{i,k}}(z)^\top \bs{g^U}(u) &&\\
		& &&w_i(x) - v_{i,k}(0,x) - 1 = q_{0_{i,k}}(x) + \bs{s^0_{i,k}}(x)^\top \bs{g_{i}^X}(x)\\
		& &&v_{i,K}(T,x) = q^T_i(x) + \bs{s^{X_T}_{i}}(x)^\top \bs{g^{X_T}_{i}}(x)\\
		& &&w_i(x) = q^w_i(x) + \bs{s^w_{i}}(x)^\top \bs{g^X_{i}}(x) \\
		& &&v_{i,k}(T_{k+1},x) - v_{i,k+1}(T_{k+1},x) = \\
		& &&\hspace{3.35cm}q_{{i,k}}^\tau(x) + \bs{s^t_{i,k}}(x)^\top \bs{g^X_{i}}(x) \\
		& && ( v_{a,k}(t,x_{a,b}) - v_{b,k}(t,x_{a,b}) ) = \\
		& &&\hspace{0.59cm} q^{1}_{k,a,b}(z) + \sum\nolimits_{j=1}^{n_X} s^{1}_{j,k,a,b}(z) h_{a,b}^\top f_j(t,x_{a,b},u)&& \\
		& &&\hspace{0.59cm} + s^{\rm a_1}_{k,a,b}(z)g^{\rm a_1}_{k,a,b}(z)\\ 
		& && ( v_{b,k}(t,x_{a,b}) - v_{a,k}(t,x_{a,b})  ) = \\
		& &&\hspace{0.59cm} q^{2}_{k,a,b}(z) - \sum\nolimits_{j=1}^{n_X} s^{2}_{j,k,a,b}(z) h_{a,b}^\top f_j(t,x_{a,b},u)&& \\ 
		& &&\hspace{0.59cm} + s^{\rm a_2}_{k,a,b}(z)g^{\rm a_2}_{k,a,b}(z),\\ 
	\end{alignedat}
\end{equation}
where \(z = [t,x,u]^\top\), \(w_i(x)\) and \(v_{i,k}(t,x)\) are polynomials, \(\mathbf{w}_i\) is a vector of coefficients of \(w_i(x)\) and \(l_i\) is a vector of Lebesgue measure moments indexed with respect to the same basis as the coefficients of $w_i$.  The decision variables in the problem are the polynomials $v_{i,k}$ and $w_i$ as well as sum-of-squares multipliers $q$, $s$ and $\bs s$.  The symbols $\boldsymbol{g_{i}^X}$, $\boldsymbol{g_{i}^{X_T}}$, $\boldsymbol{g_{i}^U}$ and $\boldsymbol{g_{k}^\tau}$ denote the column vectors of polynomials describing the sets $X_i$, $X_T\cap X_i$,  $U$ and $[T_k,T_{k+1}]$ in that order.
The degrees of all polynomial decision variables are chosen such that the degrees of all polynomials appearing in (\ref{eq:dualSOS}) do not exceed a given relaxation order $d$. This is a design parameter controlling the accuracy of the approximation.

The outer approximation of the ROA of degree \(d\) is defined by the piece-wise polynomial \(v^d\) as
\begin{equation}
	X_{d} = \{x \mid v^d(0,x) \ge 0\}.	
\end{equation} 
\corr{
	It is worth mentioning that if some splits coincide and create a ``degenerate'' subset \(X_{\rm d}\), the subset will simply lose dimensions; in the most extreme case it will become a point, 
	which is still a valid semi-algebraic set and no special treatment is necessary. 
}




\subsection{Strong Duality}
\label{sec:duality_proof}
We shall now show that the splitting does not break the strong duality property of the original, non-split ROA problem. The infinite-dimensional setting is addressed later in Theorem \ref{th:infinite_gap}. Here we focus on the strong duality of the relaxed problem \corr{\eqref{eq:dualSOS}}, for which we use the general proof from \cite{Tacchi2021}. To use the results from \cite{Tacchi2021}, we need to show that our problem can be written
as the augmented version of the generalized moment problem (GMP)
\begin{equation}
	\label{eq:GMP}
	\begin{alignedat}{2}
		p^\star_{\rm GMP}& = \sup \int \vect{c} \D{\mathbb{\vect{\mu}}} \\ 
		\text{s.t.}&  \quad \int \vect{\Phi}_{\alpha} \D{\vect{\mu}} = a_\alpha 	\quad	\alpha \in \mathbb{A} \\
		&\quad \int \vect{\Psi}_{\beta} \D{\vect{\mu}} \leq b_\beta 		\quad \beta \in \mathbb{B} \\ 
		&\quad \vect{\mu} \in \corr{\mathcal{M}(\mathbf{K}_1)_+ \times \dots \times \mathcal{M}(\mathbf{K}_N)_+,}
	\end{alignedat}
\end{equation}
where \corr{\(\mathcal{M}(\mathbf{K}_i)_+\)} is a set of positive measures supported on a set \corr{\(\mathbf{K}_i\)}, \(\mathbb{A}\) and \(\mathbb{B}\) are index sets, \(a_\alpha\) and \(b_\beta\) are scalars, \(\vect{\Phi}_\alpha = \begin{bmatrix} \Phi_{\alpha_1} & \dots & \Phi_{\alpha_{N}}
\end{bmatrix}^\top\) is a vector of polynomials and similarly for \(\vect{\Psi}_\beta\) \new{and \(\vect{c}\)}.
The vector notation is to be understood as 

\begin{equation}
	\int \vect{\Phi}_\alpha \D{{\vect\mu}} = \sum_{p=1}^N \int \Phi_{\alpha_p} \D{\mu_p}
\end{equation}

The primal of \eqref{eq:dualOriginalSPlit} reads
\begin{equation}
	\label{eq:primalSplit}
	\begin{alignedat}{2}
		&p^\star_s = &&\sup \sum_{i} \int_{X_i} 1 \D{\mu_{T_1}^i} \\ 
		&\text{s.t.  } &&\mu^i_{T_1} + \hat{\mu}^i_0 = \lambda \\
		& &&\begin{aligned}
			-\mathcal{L}' \mu^i_k 
			+&(\mu^i_{T_{k+1}} \otimes \delta_{t=T_{k+1}}) 
			- (\mu^i_{T_k} \otimes \delta_{t=T_k}) \\ 
		+&\sum_{b \in N_{X_i}^{\rm out}} (h_{i,b}^\top f)' \mu^{i \cap b}_k
		-\sum_{a \in N_{X_i}^{\rm in}} (h_{a,i}^\top f)' \mu^{a \cap i}_k
		 = 0 \\
		 & \mspace{184.5mu}   \forall(t,x,u)  \in [T_k,T_{k+1}] \times X_i \times U
		\end{aligned}\\
		&&&\begin{aligned}
		&\mu^i_k \in \mathcal{M}([T_k,T_{k+1}]\times X_i \times U)_+ &&\forall (i,k) \in \mathbb{Z}_I \times \mathbb{Z}_{K-1}\\
		 &\mu^i_{T_k} \in \mathcal{M}(X_i)_+ &&\forall (i,k) \in \mathbb{Z}_I \times \mathbb{Z}_{K-2} \\ 
		 &\mu^i_{T_{K-1}} \in \mathcal{M}(X_T)_+  &&\forall i\in \mathbb{Z}_I \\
		 &\hat{\mu}^i_0 \in \mathcal{M}(X_i)_+ &&\forall i\in \mathbb{Z}_I 
		\end{aligned}\\
		&&& \mu^{a\cap b}_k \in \mathcal{M}([T_k,T_{k+1}]\times \corr{(X_a \cap X_b)} \times U)_+ \\
		&&&\quad\quad \forall (k,(a,b)) \in \mathbb{Z}_{K-1} \times N_X,
	\end{alignedat}
\end{equation}
\corr{with decision variables \(\mu_{T_k}^i, \hat{\mu}_0^i,\mu_k^i,\text{ and }\mu_k^{a\cap b}\)}. The normal vector \(h_{a,b}\) \corr{of the boundary  \(X_a \cap X_b\) is element-wise positive.}
The set \(N_{X_i}^{\rm in}\) contains indices of neighbours of \(X_i\), such that the normal of their common boundary \(h_{\cdotp,i}\) points to \(X_i\). Similarly for \(N_{X_i}^{\rm out}\) in the opposite direction. For our specific case of splits along the axes and element-wise positive normal vectors, we can write 
\begin{equation}
	\begin{split}
		N_{X_i}^{\rm in} &= \{ 
	a \in \mathbb{Z}_I : X_i \cap X_a \neq 0,
	h_{a,i}^\top (x_i - x_a) \geq 0	
	\}\\
	N_{X_i}^{\rm out} &= \{ 
	b \in \mathbb{Z}_I : X_i \cap X_b \neq 0,
	h_{i,b}^\top (x_i - x_b) \leq 0	
	\}.
	\end{split}
\end{equation}
\begin{figure}[!htb]
	\centering
		\includegraphics{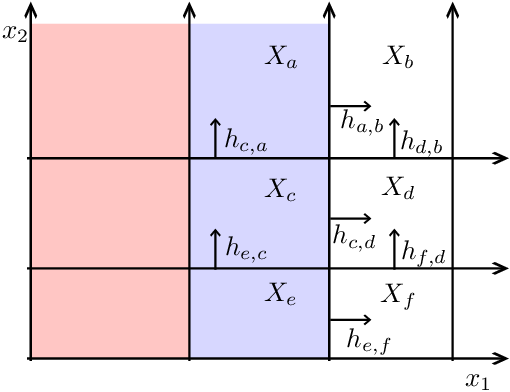}
\caption{Example of split state-space. For the set \(X_d\) it holds
\(N_{X_d}^{\rm in} = \{c,f\}\), \(N_{X_d}^{\rm out} = \{b\}\).
The blue area is the set \mbox{\(X_{\rm blue} = \bigcup_{i \in \mathcal{E}(h_{a,b},a)} X_i\)}
and both the red and blue areas together are the set  \mbox{\(
	X_{\rm red} \cup X_{\rm blue} = \bigcup_{i \in \mathcal{P}(h_{a,b},a)} X_i\)}.
	The set
	\(\mathcal{E}_{\rm n}(h_{a,b},a) \) is equal to 
	\( \{(a,b),(c,d),(e,f)\}\).
	Note that \(\mathcal{P}(h_{a,b},a) = \mathcal{P}(h_{c,d},c) = \mathcal{P}(h_{e,f},e)\), similarly for \(\mathcal{E}(\cdot,\cdot)\) and \(\mathcal{E}_{\rm n}(\cdot,\cdot)\).
	}
\label{fig:ex_normals}
\end{figure}

Considering polynomial test functions \(\Psi_1 = \Psi_1(x)\) and \(\Psi_2 = \Psi_2(t,x)\),
we can write the equality constraints of \eqref{eq:primalSplit}
as 
\begin{equation}
	\int \Psi_1 \D{\mu^i_{T_1}} + \int \Psi_1 \D{\hat{\mu}^i_{0}} = \int \Psi_1 \D{\lambda} 
\end{equation}
and 
\begin{equation}
	\begin{split}
		\int \mathcal{L} &\Psi_2 \D{\mu^i_k} + \int \Psi_2(T_{k+1},\cdotp) \D{ \mu^i_{T_{k+1}}} 
		- \int \Psi_2(T_k,\cdotp) \D{\mu^i_{T_k}} \\
		+ &\sum_{b \in N_{X_i}^{\rm out}} \int h_{i,b}^\top f \Psi_2 \D{\mu^{i\cap b}_k }
		- \sum_{a \in N_{X_i}^{\rm in}} \int h_{a,i}^\top f \Psi_2 \D{\mu^{a\cap i}_k } = 0.	
	\end{split}
\end{equation}
Since \(\int \Psi_1 \D{\lambda}\) is a constant and \(f\) is assumed to be polynomial,
we can see that \eqref{eq:primalSplit} fits the general description \eqref{eq:GMP}.

To use the results from \cite{Tacchi2021} for showing 
strong duality of the relaxation \eqref{eq:dualSOS}, we must satisfy the following
two assumptions

\begin{assumption}
	\label{ass:balls}
The description of the sets \corr{\(\mathbf{K}_i\)} in \eqref{eq:GMP}
contains the ball constraint 
\mbox{\(g_{\text{B}_i}(x) = R_{\rm B}^2 - ||x||^2\)},
such that \(\corr{\mathbf{K}_i} \subset \{x : g_{\text{B}_i}(x) \geq 0\}\).
\end{assumption}
Although the problem \eqref{eq:dualOriginalSPlit}
 does not contain the ball constraints in 
the presented form, they can be easily added without changing
the optimal value to ensure the strong duality of the problem.

Note that if the subset \(K_i\) is a hypercube, 
it is usually described by \(g_{K_i,p}(x) = R_{p}^2 - ||x_{p}||^2 \) for all \( p=1\dots n\);
in this case, the redundant ball constraint is not needed.

\begin{assumption}
	\label{ass:splitflow}
All \(h_{a,b}^\top f\) are nonzero on 
the boundary \mbox{\(X_a \cap X_b \)} for all \((a,b)\in N_X\) \corr{and \(u\in U\)}. 

\end{assumption}


\begin{lemma}\label{lem:boundedMass}
    If Assumption~\ref{ass:splitflow} is satisfied, the 
	set \(X\) is connected, and 
	\(h_{a,b}^\top f\) is nonzero on the splits,
	 then the feasible set of (\ref{eq:primalSplit}) is bounded, i.e., there exists a constant $C > 0$ such that $\int 1\,\D\mu < C$ for all measures $\mu$ appearing in~(\ref{eq:primalSplit}).
\end{lemma}
\begin{proof}
     To prove that all masses are bounded, we first sum the equations of~(\ref{eq:primalSplit}) over time and space, which eliminates the additional
boundary measures introduced by the time and space splits respectively. This will give us a similar situation as in the original work \cite{milanroa}
and simplify the proof for the boundary measures. \corr{We shall denote \(\int 1 \D{\mu}\) as \(\text{mass}(\mu)\).} 

The first constraint \(\mu^i_{T_1} + \hat{\mu}^i_0 = \lambda\) implies that \(p^\star_s\) is bounded and \(\mu^i_{T_1}\) and \(\hat{\mu}^i_{0}\) are bounded.

By summing the second constraint over all \(i\) and \(k\) we obtain
\begin{equation}	
	\label{eq:proof_sum_i_k}
  \sum_i \mu^i_{T_{K-1}} \otimes \delta_{t=T_{K-1}}   = \sum_{i,k} \mathcal{L}' \mu^i_k + \sum_i \mu^i_{T_1} \otimes \delta_{t=0}
\end{equation}
A test function \(\Psi(t,x) = 1\) then gives us
\begin{equation}
	 \sum_i \text{mass}(\mu^i_{T_{K-1}}) = \sum \text{mass}(\mu^i_{T_1}),
\end{equation}
meaning that \(\mu^i_{T_{K-1}}\) are bounded.

With a test function \(\Psi(t,x) = t\) we obtain 
\begin{equation}
	T \sum_i \text{mass}(\mu^i_{T_{K-1}}) = \sum_{i,k} \text{mass}(\mu^i_k)
\end{equation}
which shows that \(\mu^i_k\) are bounded.

Let us sum the second constraint over all \(i\) for some time \(k = s\)
\begin{equation}
	\begin{split}
		-\sum_i \mathcal{L}' \mu^i_s 
			+\sum_i(\mu^i_{T_{s+1}} \otimes \delta_{t=T_{s+1}}) 
			- \sum_i(\mu^i_{T_s} \otimes \delta_{t=T_s}) =0,
	\end{split}
\end{equation}
with a test function \(\Psi(t,x) = 1\) we obtain
\begin{equation}
	\label{eq:sum_ks}
			\sum_i\text{mass}(\mu^i_{T_s}) = 
			\sum_i \text{mass}(\mu^i_{T_{s+1}}) .
\end{equation}
Since \(\mu_{T_1}^i\) are bounded, \(\mu_{T_2}^i\) are bounded as well; by induction all \(\mu_{T_s}^i\) are bounded.

To show the boundedness of the remaining measures  \(\mu_k^{a \cap b}\) let us first introduce the following sets of indices: 
\begin{equation}
	\mathcal{P}(h,i) := \{p \in \mathbb{Z}_I | \forall x_i \in X_i \quad
	\exists x_p \in X_p: h^{\top}x_p \leq h^\top x_i\},
\end{equation}
\begin{equation}
	\label{eq:set_Eps}
	\mathcal{E}(h,i) := \{p \in \mathbb{Z}_I | \forall x_i \in X_i \quad
	\exists x_p \in X_p: h^{\top}x_p = h^\top x_i\}.
\end{equation}

By summing the second constraint over the indices \(\mathcal{P}(h, i)\), we will be left only with the boundary measures \(\mu^{\cdot \cap : }_k\) corresponding to the normal vectors in the direction of \(h\), their indices are characterized by the set \(\mathcal{E}(h, i)\). 
For convenience, we also define the set of tuples
\begin{equation}
	\label{eq:set_Eps_n}
	\mathcal{E}_{\rm n}(h,i) := \{(i,j) \in \mathbb{Z}_I^2 | \forall i \in \mathcal{E}(h,i) \quad
	\exists X_j : h_{i,j} = h\},
\end{equation}
which contains the set indices \(\mathcal{E}(h,i)\) and indices of their connecting neighbours in the direction \(h\). See Fig. \ref{fig:ex_normals} for graphical representation.

Let us fix time as \(k=s\) and focus only on space. We shall show the boundedness of a measure \(\mu_s^{a \cap b}\) which corresponds to the normal vector \(h_{a,b}\). Let us sum the equations corresponding to time \(s\) and the space indices \(\mathcal{P}(h_{a,b},a)\), with a test function \(\Psi(t,x) = 1\) we obtain 
\begin{equation}
	\label{eq:proof_flow_measure1}
	\begin{split} 
		0 + \sum_{i \in \mathcal{P}(h_{a,b},a)}
		&\Bigl( \text{mass}(\mu_{T_{s+1}^i}) - \text{mass}(\mu_{T_{s}^i}) \Bigr) \\
	&+ \sum_{(i,j) \in \mathcal{E}_{\rm n}(h_{a,b},a)} \int h_{i, j}^\top f \D{\mu^{i \cap j}_s}
	= 0,
	\end{split} 
\end{equation}
where all the normal vectors have the same direction, i.e. \(h_{i, j} = h_{a,b}\).
The normal vectors with directions different from \(h_{a,b}\) have been summed out since they all appear in \(\mathcal{P}(h_{a,b}, a)\) exactly twice with opposite signs (once as incoming and once as outgoing vector).
 We can rewrite the equation \eqref{eq:proof_flow_measure1} as 
\begin{equation}
	\label{eq:proof_flow_measure2}
	\begin{split} 
		0 + \sum_{i \in \mathcal{P}(h_{a,b},a)}
		&\Bigl( \text{mass}(\mu_{T_{s+1}^i}) - \text{mass}(\mu_{T_{s}^i}) \Bigr) \\
	&+  \int h_{a,b}^\top f \D{\mu_{s}^{\mathcal{E}}}
	= 0,
	\end{split} 
\end{equation}
where
\begin{equation}
	\label{eq:compound_measure}
	\mu_s^{\mathcal{E}} = \sum_{(i,j) \in \mathcal{E}_{\rm n}(h_{a,b},a)} \mu_s^{i \cap j}.
\end{equation}
Since the top part of \eqref{eq:proof_flow_measure2} is bounded, the integral \(\int h_{a,b}^\top f \D{\mu_{s}^{\mathcal{E}}}\) is also bounded.
The measure \(\mu_s^{\mathcal{E}}\) is supported on a intersection of \(X\) and a hyperplane with normal vector \(h_{a,b}\).
Using the assumptions on connected \(X\) and that \(h_{a,b}^\top f\) is nonzero on the support of \(\mu_s^{\mathcal{E}}\), we can conclude that \(\mu_{s}^{\mathcal{E}}\) is bounded.

This implies boundedness of all the measures \( \mu_s^{i \cap j}\) from \eqref{eq:compound_measure}, since they all have the same sign. Due to the construction of the set \(\mathcal{E}(h_{a,b},a)\), the measure \(\mu^{a\cap b}_s\) is trivially one of the measures \(\mu^{i \cap j}_s\) and is therefore bounded.

This procedure can be done for all the measures \(\mu^{a \cap b}_{k}\).

We can now conclude that under the assumption of 
\begin{equation}
(h_{a,b}^{\top}f)(t,x,u) \neq 0 \quad \text{on } [0,T] \times X_{a,b} \times U \quad \forall (a,b) \in N_X,
\end{equation}
the feasible set of \eqref{eq:primalSplit} is bounded.
\end{proof}
\begin{theorem}
	\label{th:infinite_gap}
    If Assumption~\ref{ass:splitflow} \new{holds} then there is no duality gap between \eqref{eq:dualOriginalSPlit} and \eqref{eq:primalSplit}, i.e., $d_s^\star = p_s^\star$.
\end{theorem}
\begin{proof}
    The proof is based on classical infinite-dimensional LP duality theory result \cite[Theorem 3.10]{nash1987linear}, following the same arguments as in \cite[Theorem 2]{milanroa}. The key ingredients to these arguments are the boundedness of masses established in Lemma~\ref{lem:boundedMass} and the continuity of the operators $\mathcal{L}'$ appearing in (\ref{eq:primalSplit}) that holds trivially.
\end{proof}
\corr{
\begin{lemma}
	\label{lem:strong_dual_relaxed}
	The SOS relaxation \eqref{eq:dualSOS} of \eqref{eq:dualOriginalSPlit} and its dual, the moment relaxation (not presented) of \eqref{eq:primalSplit}, have zero duality gap if 
	Assumptions \ref{ass:balls} and \ref{ass:splitflow} hold.
\end{lemma}
\begin{proof}
    The proof follows directly from \cite[Proposition 6]{Tacchi2021}, where the only missing part is boundedness of the masses of the relaxed problem, which follows from boundedness of masses of \eqref{eq:primalSplit} \new{since the proof of the boundedness of the masses of \eqref{eq:primalSplit} in Lemma~\ref{lem:boundedMass} uses only constant or linear test functions.}
\end{proof}
}

\section{ SDP Differentiation }
\label{sec:sdp_diff}
We will consider the SDP in the form of a primal-dual pair, parametrized by \(\theta\). To stay consistent with the usual notation, we shall abuse ours and use \(x\) as a vector of decision variables in the context of semidefinite programming.
\begin{equation}
	\label{eq:sdp_prim_dual}
	\begin{alignedat}{2}
	p^\star(\theta) = \min c(\theta)^\top x \\ 
	\text{s.t. } A(\theta)x + s= b(\theta)\\
	s \in \mathcal{K}
	\end{alignedat}
	\quad
	\begin{alignedat}{2}
		d^\star(\theta) = \min b(\theta)^\top y \\ 
		\text{s.t. } A^\top(\theta)y + c  = 0\\
		y \in \mathcal{K}^\star,
		\end{alignedat}
\end{equation}
with variables \(x \in \mathbb{R}^{n} \), \(y \in \mathbb{R}^{m}\), and \(s \in \mathbb{R}^{m}\)
with data \(A \in \mathbb{R}^{m\times n} \), \(b \in \mathbb{R}^{m} \), and \(c \in \mathbb{R}^{n} \).
We can assume strong duality due to the \corr{Lemma \ref{lem:strong_dual_relaxed} }and therefore \(p^\star = -d^\star\). 
The KKT conditions are 
\begin{equation}
	Ax +s = b,A^\top y + c = 0,s^\top y = 0,s \in \mathcal{K}, y \in \mathcal{K}^\star.
\end{equation}
Notice that the conic constraints do not depend on \(\theta\), this reflects the fact that neither the number of the split regions \(X_i\) nor the number of the time splits \(T_k\) changes. 

We can describe the (primal) SDP concisely as a function of \(\theta\) as
\begin{equation}
\label{eq:solution_fun_diff}
	p^\star(\theta) = \mathcal{S}(A(\theta),b(\theta),c(\theta)) = \mathcal{S}(\mathcal{D}(\theta)),
\end{equation}
where \(\mathcal{D}(\theta)\) is a shorthand for all the program data depending on \(\theta\). The goal of this paper is to find a (sub)optimal set of parameters \(\theta^\star\) such that
\begin{equation}
	\label{eq:argmin_sdp}
	\theta^\star = \argmin_\theta \mathcal{S}(\theta).
\end{equation}
\corr{Note that we are dealing with multiple meanings for the optimal value \(p^\star\); let us clarify that \(p^\star(\theta)\) is the minimal objective value of \eqref{eq:sdp_prim_dual} for \emph{some} parameters \(\theta\), while \(p^\star(\theta^\star)\) is the minimal objective value for the optimal parameters \(\theta^\star\) which is what we are after. In the context of ROA, \(p^\star(\theta^\star)\) corresponds to the ROA with optimal splits.}
\corr{
We shall tackle \eqref{eq:argmin_sdp} by 
assuming differentiability of \(\mathcal{S}\)
which will be rigorously proven later, in Lemma \ref{lem:differentiability}.}
Assuming an existing gradient, we can search for \(\theta^\star\) via a first-order method, which iteratively
updates the parameters \(\theta\) by using the gradient of \(s\) as
\begin{equation}
	\theta_{k+1} = \theta_{k} - \gamma \nabla \mathcal{S}(\mathcal{D}( \theta))
\end{equation}
for some initial guess \(\theta_0\)
(e.g. obtained from the recommendations in \cite{Cibulka2022})
and a stepsize \(\gamma\).
The stepsize can be a function of \(k\) and/or
some internal variables of the concrete gradient descent algorithm.
The following section will present two ways of calculating \(\nabla \mathcal{S}(\mathcal{D}( \theta))\).

The parameter \(\theta\) is considered to be a column vector of 
size \(n_\theta\). The perturbation of the vector \(\theta\) in direction \(k\) is 
\begin{equation}
	\label{eq:phi_perturbation}
 \theta + \epsilon_\theta e_k = \begin{bmatrix}
		\theta_1 \\ \vdots \\ \theta_k + \epsilon_k \\  \vdots \\ \theta_{n_\theta}
	\end{bmatrix}
\end{equation}
where \(e_k\) is the \(k\)th vector of standard base, containing 
\(1\) at \(k\)th coordinate and \(0\) everywhere else. 
The scalar \(\epsilon_k\) is the perturbation size.

The object \(\mathcal{D}(\theta)\) is to be understood as a vector of 
all the \corr{SDP data,} 
 for example
\begin{equation}
	\mathcal{D} = [\new{A_{1,1},\dots,A_{m,n}},b_1,\dots,b_m,c_1,\dots,c_n]^\top,
\end{equation}
where we dropped the dependence on \(\theta\) to lighten up the notation.

\subsection{Methods for finding the derivative}
\label{sec:methods_diff}

\subsubsection{Finite differences}
\label{sec:finitediff}
The estimate of the \new{derivatite} of \eqref{eq:solution_fun_diff}
at the point \(\theta\)	in the direction \(e_k\) 
is 
\begin{equation}
	\frac{\Delta \mathcal{S} }{\Delta \theta_k } = 
	\frac{\mathcal{S} (\mathcal{D}(\theta + \epsilon_\text{f}e_k)) - \mathcal{S}(\mathcal{D}(\theta))}{\epsilon_\text{f}},
\end{equation}
where the step \(\epsilon_{\rm f}\) is a free parameter.
The gradient is then estimated as 
\begin{equation}
	\Delta \mathcal{S} = \begin{bmatrix}
		\frac{\Delta p^\star }{\Delta \theta_1 } & \dots &
		\frac{\Delta p^\star }{\Delta \theta_{n_\theta} }.
	\end{bmatrix}
\end{equation}

\subsubsection{Analytical \new{derivative}}
The gradient of \(\mathcal{S}\) can be written as
\begin{equation}
	\nabla \mathcal{S}(\mathcal{D}( \theta))
	= \frac{\D  \mathcal{S}(\mathcal{D})}{\D  \theta} = 
	\frac{\D  \mathcal{S}}{\D  \mathcal{D}} \frac{\D  \mathcal{D}}{\D  \theta}.
\end{equation}
The first fraction \(\frac{\D  \mathcal{S}}{\D  \mathcal{D}}\)
signifies how the problem solution changes with respect to the input data.
This problem has been tackled in \cite{diff_sdp2} for 
general conic programs; here we shall address some specific issues tied to SOS-based SDPs.

The second fraction \(\frac{\D{\mathcal{D}}}{\D{\theta}}\) 
shows how the problem data changes with the parameters \(\theta\). 
Modern tools (YALMIP \cite{yalmip},\cite{Lofberg2009}, GloptiPoly \cite{GloptiPoly}, 
The sum of Squares Programming for Julia \cite{weisser2019polynomial},\cite{legat2017sos}) allows the user to write directly the polynomial problem \eqref{eq:dualOriginalSPlit} or its SOS representation,
alleviating the need for constructing the problem data \((A,b,c)\) directly. Despite the undeniable advantages this abstraction brings, it makes it more difficult to work on the problem data directly, since these parsers are usually not created to be autodifferentiable. For this reason, we shall estimate the derivatives \(\D{\mathcal{D}}\) numerically, striking a tradeoff between convenience and accuracy.

For example, sensitivity to \(\theta_k\) is obtained as
\begin{equation}
	\frac{\partial \mathcal{D}}{\partial \theta_k} = \frac{1}{|P_{\epsilon}|} \sum_{\epsilon \in P^k_\epsilon}
	\frac{\mathcal{D}(\theta + \epsilon e_k) - \mathcal{D}(\theta)}{\epsilon},
\end{equation}
where \(P^k_\epsilon = \{-\epsilon_d, \dots, \epsilon_d \}\) is 
a set of perturbation steps sizes of the parameter \(\theta\) in the direction \(k\). Each evaluation of \(\mathcal{D}\) is to be understood
as a call to one of the aforementioned programming tools which constructs \eqref{eq:sdp_prim_dual} from \eqref{eq:dualOriginalSPlit}.

The following subsection will explain how to obtain the derivative 
\(\frac{\D s}{\D \mathcal{D}}\)
\subsubsection*{Obtaining \(\frac{\D{\mathcal{S}}}{\D{\mathcal{D}}}\) }
This subsection summarizes the approach 
listed in \cite{diff_sdp2} and \cite{Busseti2018},
while focusing on the specific case of the SOS-based SDPs. 
The generic approach presented in 
\cite{diff_sdp2} is not immediately usable for our concrete problem,
therefore we shall provide some remedies in the following subsections.

We shall first quickly review the generic approach in \cite{diff_sdp2}.
In the following text, \(\Pi_{\mathcal{A}}\) shall denote a projection onto the set \(\mathcal{A}\)
and \(\Pi\) a projection onto \(\mathbb{R}^n \times \mathcal{K}^\star \times \mathbb{R}_{+}\).
We shall also drop the dependence on \(\theta\) to lighten up the notation. 
Lastly, we abuse our notation again 
and use \(v\) and \(w\) to denote vectors corresponding
to the primal-dual solution of \eqref{eq:sdp_prim_dual}
in the context of semidefinite programs.

The derivative of the solution can be written as
\begin{equation}
	\D{(\mathcal{S})} = \D{(c^\top x)} = 
	\D{c^\top} + c^\top \D{x} = \D{b^\top}y + b \D{y},
\end{equation}
where the primal-dual derivatives are obtained as 
\begin{equation}
	\begin{bmatrix}
		\D{x} \\ \D{y}
	\end{bmatrix} = 
	\begin{bmatrix}
		\D{u} - x^\top \D{w} \\ 
		\D{\Pi}_{ \mathcal{K}^\star} (v)^\top \D{v} - y^\top \D{w}
	\end{bmatrix},
\end{equation}
where the variables \(u,v,\textrm{ and } w\) are 
related to the solution by 
\begin{equation}
	z = \begin{bmatrix}
		x \\ y-s \\ 1 
	\end{bmatrix} = \begin{bmatrix}
		u \\ v \\ w
	\end{bmatrix}.
\end{equation}
Note that is \(w\) a normalization parameter 
which is in our case always equal to \(1\), and thus not necessary;
we only keep it to stay consistent with \cite{diff_sdp2}. 
The meaning and other possible values of \(w\) are explained
in \cite{Busseti2018}.
The derivative of \(z\) is obtained as the solution to
 \begin{equation}
	\label{eq:Mdz_equation}
	M \cdotp \D{z} = g,
 \end{equation}
where \(M = ((Q - I)\D{\Pi}(z) + I )/ w\) and 
\(g = \D{Q} \cdotp \Pi(z/|w|)\). Note that the matrix \(M\) depends only on 
the current solution, not the perturbations; we shall exploit it later in this section.
The matrices \(Q\) and \(\D{Q}\) are defined as 
\begin{equation}
	Q = \begin{bmatrix}
		0 & A^\top & c\\
		-A & 0 & b \\
		-c^\top & -b^\top & 0
	\end{bmatrix},
	\D{Q} = \begin{bmatrix}
		0 & \D{A}^\top & \D{c}\\
		-\D{A} & 0 & \D{b} \\
		-\D{c}^\top & -\D{b}^\top & 0
	\end{bmatrix}.
\end{equation}

Let us now write relevant cone projections and their derivatives:

\begin{equation}
	\Pi_{\mathbb{R}^n}(x) = x \quad \forall x \in \mathbb{R},
\end{equation}
\begin{equation}
	\D{\Pi_{\mathbb{R}^n}}(x) = 1 \quad \forall x \in \mathbb{R},
\end{equation}
\begin{equation}
	\Pi_{\{0\}}(x) = 0 \quad \forall x \in \mathbb{R},
\end{equation}
\begin{equation}
	\D{\Pi_{\{0\}}}(x) = 0 \quad \forall x \in \mathbb{R},
\end{equation}
\begin{equation}
	\Pi_{\mathbb{R}_{+}}(x) = \max(x,0) \quad \forall x \in \mathbb{R},
\end{equation}
\begin{equation}
	\D{\Pi_{\mathbb{R}_{+}}}(x) = \frac{1}{2}(\mathrm{sign}(x) + 1) \quad \forall x \in \mathbb{R}\setminus \{0\},
\end{equation}
\begin{equation}
	\Pi_{\mathbb{S}_{+}}(X) = U \Lambda_{+} U^\top \quad \forall X \in \mathbb{S}^{r\times r},
\end{equation}
where \(X = U \Lambda U^\top\) is the eigenvalue decomposition of $X$, i.e., \(\Lambda\) is a diagonal matrix  of the eigenvalues of \(X\) and $U$ an orthonormal matrix. The matrix \(\Lambda_{+}\) is obtained as 
\(\max(\Lambda,0)\), element-wise.

Finally, the derivative of \({\Pi_{\mathbb{S}_{+}}}\) at a non-singular point \(X\) in the direction \(\tilde{X} \in \mathbb{R}^{r\times r}\) is
\begin{equation}
	\D{\Pi}_{\mathbb{S}_{+}}(X)(\tilde{X}) = U(B \circ (U^\top \tilde{X} U))U^\top,
\end{equation}
where \(\circ\) is element-wise product and
\begin{equation}
	B_{i,j} = \begin{cases}
		0 & \textrm{ for } i \leq k, j \leq k\\
		\frac{|\lambda_i|}{|\lambda_i| + |\lambda_j|} & \textrm{ for } i > k, j\leq k\\
		\frac{|\lambda_j|}{|\lambda_i|+|\lambda_j|} & 
		\textrm{ for } i \leq k, j > k
		\\
		1 & \textrm{ for } i > k, j > k,
	\end{cases}
\end{equation}
where \(k\) is the number of negative eigenvalues of \(X\),
and \(U\) is chosen such that the eigenvalues \(\lambda_i\) in the diagonal matrix \(\Lambda\) in the  decomposition \(X = U \Lambda U^\top\) are sorted in
increasing order, meaning that the first \(k\) eigenvalues are negative.



\subsubsection*{Exploiting problem structure}
The most demanding task in this approach is solving
\begin{equation}
	\label{eq:leastsquare_original}
	M \cdotp \D{z} = g
\end{equation}
for \(\D{z}\). The paper \cite{diff_sdp2} suggests the use of LSQR \cite{Paige1982} instead of direct solve via factorization when the matrix \(M\) is too large to be stored in dense form.

Luckily, we can also factorize the matrix \(M\) in its sparse form, via free packages such as SuiteSparse \cite{suitesparse} (the default factorization backend in MATLAB \cite{MATLAB} and Julia \cite{bezanson2017julia}), Intel MKL Pardiso \cite{mklpardiso}, and MUMPS \cite{MUMPS:1},\cite{MUMPS:2}. Moreover, recall that we have \(n_{\theta}\) parameters and have to solve \eqref{eq:leastsquare_original} in \(n_{\theta}\) directions
\corr{resulting in
\begin{equation}
	\label{eq:leastsquare_original_multiple}
	M \cdotp \D{z} = [g_1,g_2,\dots, g_{n\theta}].
\end{equation}}
Since the matrix \(M\) does not depend on the perturbed data, we need to factorize it only once to solve \eqref{eq:leastsquare_original} for all \(n_{\theta}\) directions of \(\theta\).

In this paper, we factorize \(M\) by the QR 
factorization \cite{QR1},\cite{QR2} as
\begin{equation}
	M = Q R,
\end{equation}
where $Q$ is orthonormal and \(R\) is upper triangular. The equation \eqref{eq:leastsquare_original} then becomes
\begin{equation}
\begin{alignedat}{2}
QR \cdotp \D{z} = g \\ 
R \cdotp \D{z} = Q^\top g, 
\end{alignedat}
\end{equation}
which is solved by backward substitution due to \(R\) being a triangular matrix. 


All the aforementioned methods (Finite differences, LSQR, and QR) are compared in Section \ref{sec:comparison}.

\subsection{Conditions of differentiability}
As was mentioned above, the analytical approach for obtaining the derivative is preferable to the finite differences. However, the analytical approach also assumes differentiability of the ROA problem with respect to the split positions, which is proved in the following Lemma.
{We use the notion of genericity from \cite[Definition 19]{complementarity_nondegen}. 
We call a property \(\mathit{P}\) of an SDP  \emph{generic} if it holds for Lebesgue almost all parameters $(A,b,c)$ of the SDP. In other words, the property fails to hold on a set of zero Lebesgue measure.\corr{Concretely, we will use the genericity of uniqueness 
\cite[Theorems 7,10, and 14]{complementarity_nondegen} and strict complementarity \cite[Theorem 15]{complementarity_nondegen} of the primal-dual solutions to \eqref{eq:sdp_prim_dual}.}
}
\begin{lemma}
	\label{lem:differentiability}
	\corr{The mapping from the split positions to the infimum of the SOS-relaxation \eqref{eq:dualOriginalSPlit} is differentiable
	at a point \({\theta}\) if
	 assumptions \ref{ass:balls} and \ref{ass:splitflow} hold,
	 and the
	 primal-dual solution of \eqref{eq:sdp_prim_dual}
	is
	 unique and strictly complementary
	for the problem data \(\mathcal{D}(\theta)\).}
	
\end{lemma}
\begin{proof}
	The conditions of differentiability according to \cite{diff_sdp2} are uniqueness 
	of the solution and differentiability of the projection \(\Pi\) of the vector \(z\), needed
	for construction of \eqref{eq:Mdz_equation}.
	Since the uniqueness is assumed, only the 
	projection \(\Pi\) needs to be investigated.

		Assuming \((x,y,s)\) to be the optimal
	primal-dual solution,
	the projection \(\Pi(z)\) can be written as
	\begin{equation}
		\Pi(z) = \begin{bmatrix}
			\Pi_{\mathbb{R}^{n}}(x) \\ 
			\Pi_{\mathcal{K}^{\star}}(y-s) \\ 
			\Pi_{\mathbb{R}_{+}}(w) 
		\end{bmatrix},
	\end{equation}
	where \(\Pi_{\mathbb{R}^{n}}\) is differentiable everywhere and \(\Pi_{\mathbb{R}_{+}}\)
	is also differentiable since we are at the solution with \(w=1\).
	The only cause for concern is \(\Pi_{\mathcal{K}^{\star}}\), where \(\mathcal{K}^\star\) 
	is a product of the positive semidefinite cone \(\mathbb{S}_{+}\) and the free/zero cone.
	Therefore \(\Pi_{\mathcal{K}^\star}\) is differentiable if and only if \(\Pi_{\mathbb{S}_{+}}\) is differentiable.

	Let us denote the semidefinite parts of \(y\) and \(s\) 
	as matrices \(Y\) and \(S\) respectively.
	The matrices \(Y\) and \(S\) commute, since \(YS = SY = 0\). 
	They also share a common set of eigenvectors \(Q\) such that \(Q^\top Q = I\),
	making them simultaneously diagonalizable as 
	\begin{equation}
		Y = Q \Lambda_Y Q^\top
	\end{equation}
	\begin{equation}
		S = Q \Lambda_S Q^\top,
	\end{equation}
	where \(\Lambda_Y\) and \(\Lambda_S\) are diagonal matrices with eigenvalues on the diagonal.
	The product \(YS\) can be then written as 
	\begin{equation}
		YS = Q \Lambda_Y Q^\top Q \Lambda_S Q^\top = Q \Lambda_Y \Lambda_S Q^\top,
	\end{equation}
	and for \(i^{\rm th}\) eigenvalue we get the condition
\begin{equation}
	\label{eq:slackness}
	\lambda^i_s \lambda^i_y = 0.
\end{equation}
	 Strict complementarity of the SDP solution means that the ranks of 
	\(Y\) and \(S\) sum up to full rank.
	Taking the sum, we can write 
	\begin{equation}
		Y + S  = Q (\Lambda_Y +  \Lambda_S) Q^\top
	\end{equation}
	and since \(\Lambda_Y\) and \(\Lambda_S\) are diagonal, we can claim that 
	\begin{equation}
		\label{eq:strictcomp}
		\lambda^i_s + \lambda^i_y \neq 0,
	\end{equation}
	 otherwise the rank of \(Y+S\) would decrease.

	By putting \eqref{eq:slackness} and \eqref{eq:strictcomp} together, we conclude 
	that for \(i^{\rm th}\) eigenvalues \(\lambda^i_s\) and \(\lambda^i_y\),
	one has to be zero and the other nonzero. 
	This implies that the matrix \(Y-S\) will not be singular and thus the 
	projection \({\Pi}_{\mathbb{S}_{+}}(Y-S)\) is differentiable, and therefore the whole SDP \eqref{eq:sdp_prim_dual} is differentiable.
\end{proof}




\subsection{Comparison of differentiation approaches}
\label{sec:comparison}
The Figure \ref{fig:comp_qr_lsqr_fin} shows 
scaling of the proposed methods with an increasing number of parameters
and the Figure \ref{fig:comp_qr_lsqr_fin_fixed_param} investigates 
their scaling 
with the degree of approximation.

We see that using 
QR factorization to solve \eqref{eq:Mdz_equation} clearly outperforms both LSQR, suggested in \cite{diff_sdp2}, and Finite differences \ref{sec:finitediff}.
We see that QR is preferable, since the factorization is 
done only once for all parameters whereas LSQR \corr{needs to solve 
\eqref{eq:leastsquare_original_multiple} \(n_{\theta}\)-times}, similarly for Finite differences which solves 
the ROA for each parameter individually.

The concrete software packages used are 
Krylov.jl \cite{montoison-orban-krylov-2020} for LSQR and SuiteSparse \cite{suitesparse} for QR factorization, both 
used through their interfaces to the programming language Julia
\cite{bezanson2017julia}.


\begin{figure}[!htb]
	\centering
	\begin{flushright}
		\inputnamedtex{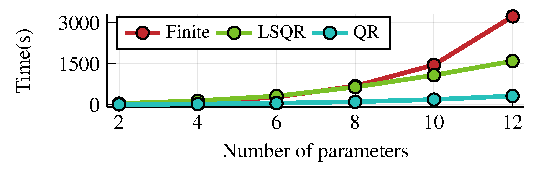}
	\end{flushright}
\caption{Computation time needed to obtain derivatives 
for degree 6 double integrator with respect to the number of parameters. The LSQR method always reached 
the maximum number of iterations, which was set to 1000.
The times for LSQR and QR include the cost of obtaining the 
matrices \(M\) and \(g\) in \eqref{eq:Mdz_equation}.
}
\label{fig:comp_qr_lsqr_fin}
\end{figure}

\begin{figure}[!htb]
	\centering
	\begin{flushright}
		\inputnamedtex{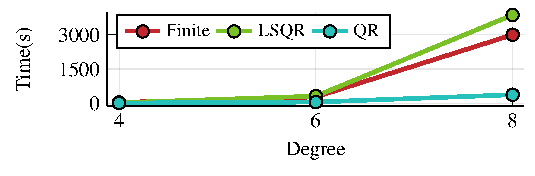}
	\end{flushright}
\caption{Computation time needed to obtain derivatives 
for double integrator with 6 parameters and increasing degree. 
 The LSQR method always reached 
the maximum number of iterations, which was set to 1000.
The times for LSQR and QR include the cost of obtaining the 
matrices \(M\) and \(g\) in \eqref{eq:Mdz_equation}.
}
\label{fig:comp_qr_lsqr_fin_fixed_param}
\end{figure}




\section{Numerical examples}
\label{sec:numerical_examples}
This section presents the optimization results 
on Double integrator and Brockett integrator,
the optimization results are presented in \ref{sec:opt_results}.
The section is divided into
optimization of \emph{low degree} and \emph{high degree}
problems, the degree is to be understood as the 
degree of the polynomial variables in \eqref{eq:dualOriginalSPlit}.

All of the examples use ADAM \cite{adam} as the first-order method for the optimization;
the stepsize of ADAM was \(0.05\) and the decay rates \(0.8\) and \(0.9\). \new{The results were obtained on a computer with 3.7GHz CPU and 256GB RAM.}

To simplify the following, let us define \(\vect{\theta}_{d}\) as the parameter path 
obtained by optimizing degree \(d\) problems, \corr{i.e. \(\vect{\theta}_{d}\) contains the split locations of each iteration of the optimization algorithm.}
Similarly, 
\(p^\star_d(\vect{\theta})\) will denote the vector  of optimal values for \new{ degree \(d\) approximation} calculated along \(\vect{\theta}\). For example, \(p^\star_6(\vect{\theta}_4)\) denotes a vector of optimal values of degree 6 approximation, evaluated on parameters obtained from \new{optimizing the split locations of a degree 4 approximation}.

\subsection{Low degree}
\label{sec:opt_results}
\subsubsection*{Double integrator}
\begin{samepage}
	First, we consider the double integrator example 
	from \cite[9.3]{milanroa} which is defined as
	\[\dot{x}_1 = x_2, 
	\  \dot{x}_2 = u\] with 
   \(X = [-0.7\times 0.7]\times[-1.2 \times 1.2]\), \(X_T = \{0\}\) and \(T = 1\).
\end{samepage}
Figure \ref{fig:descent_Double_d4np_4}
shows the results for optimization of the \new{split locations with degree 4 polynomials}.
The initial conditions were equidistantly placed split positions. The dotted black line shows the estimated global optimum, which was attained by sampling the parameter space
on a grid of square cells with sizes of \(0.1\). A total of 25116 unique split positions were evaluated in 23 hours.

The optimizer attained by the proposed method was
\(\theta^{\star}  =  \begin{bmatrix}
	-0.059 & 0.070 & -0.017 & 0.015
   \end{bmatrix}\) while the global estimate was at 
   \({\theta}^{\star}_{\rm g} = \begin{bmatrix}
	0&0.2&-0.4&0
   \end{bmatrix}\). 

The obtained optimum improves the initial guess by \(58\%\) with respect to the global optimum estimate, and it was found in 2 minutes whereas the global estimate took 23 hours.

\begin{figure}[!htb]
	\centering
	\begin{flushright}
		\inputnamedtex{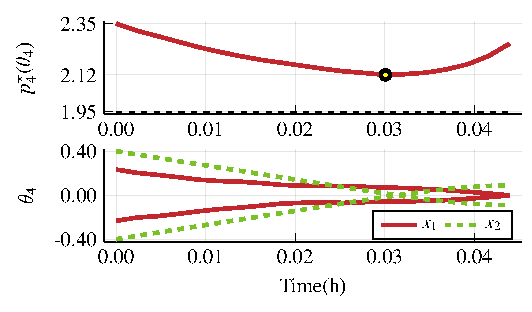}
	\end{flushright}
\caption{\new{Double integrator: degree 4 polynomials and 4 splits}. \new{The optimal values of the SDP relaxation of the ROA are shown in the top plot.} The bottom plot shows how the split positions evolved during the optimization process. There were two splits for each state variable. The black-and-yellow dot represents the attained minimum. The estimated global optimum is shown by a black dotted line.}
\label{fig:descent_Double_d4np_4}
\end{figure}

The Figure \ref{fig:valuefun_double} shows 1-dimensional line segment 
parametrized by \(r \in [0,2]\), connecting 
 the attained solution \(\theta^\star\) (\(t=0\)) to the global estimate \(\theta^\star_{\rm g}\) (\(t=1\)). 
We see that in this particular direction, the optimum is quite sharp, which makes it 
difficult to find precisely using gridding. 
\begin{figure}[!htb]
	\centering
	\begin{flushright}
		\inputnamedtex{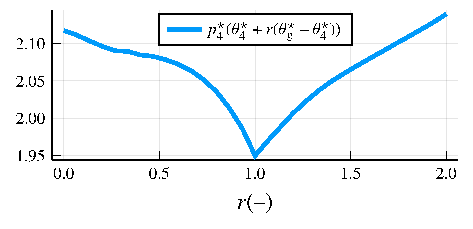}
	\end{flushright}
\caption{\new{Double integrator: Double integrator: degree 4 polynomials and 4 splits.} Slice of the value function  between the 
attained optimum (\(r=0\)) and the estimated global optimum (\(r=1\)).}
\label{fig:valuefun_double}
\end{figure}

\subsubsection*{Brockett integrator}
The Brockett integrator is defined according to 
\cite{Brockett83asymptoticstability} as
\begin{equation}
\begin{alignedat}{2}
\dot{x}_1 &= u_1\\
\dot{x}_2 &= u_2\\
\dot{x}_3 &= u_1x_2 - u_2x_1,
\end{alignedat}
\end{equation}
where \(X = \{x \in \mathbb{R}^3 : ||x||_\infty \leq 1\}\), \(X_T = \{0\}\), \mbox{\(U = \{u \in \mathbb{R}^2 : ||u||_2 \leq 1\}\)}, and \(T = 1\).
This system usually serves as a benchmark 
for nonholonomic control strategies, because it is one of the simplest systems for which 
there exists no continuous control law which would make the origin asymptotically
stable \cite{Brockett83asymptoticstability}.

Figure \ref{fig:descent_Bro_d4np_6} shows the optimization 
results for degree 4 approximation. The estimate of the global optimum was attained by sampling the parameter space on a grid with cell size \(0.1\). Furthermore, we assumed that the splits are symmetrical along all three axes, making the search space 3-dimensional. The computation time of the sampling was 16 hours over 1000 unique split positions. Without the symmetry assumption, the computation would be intractable as the full search space has 6 dimensions. The optimizer attained by our method was
  \begin{equation*}
	\theta^{\star} =[0.011 , -0.011 , 0.011 , -0.011 , 0.004 , -0.004]^{\top},
  \end{equation*}
while the global estimate was
\begin{equation*}
	{\theta}^{\star}_{\rm g} = [0,0,0,0,0,0]^{\top}.
\end{equation*} 
Both minimizers are quite close in this case. The found optimum improves the initial guess by \(62\%\) with respect to the global optimum estimate, and it was found in 30 minutes whereas the global estimate with symmetry assumption took 16 hours. A brute-force search in the whole parameter space would take roughly 12 years on the same computational setup.

\begin{figure}[!htb]
	\centering
	\begin{flushright}
		\inputnamedtex{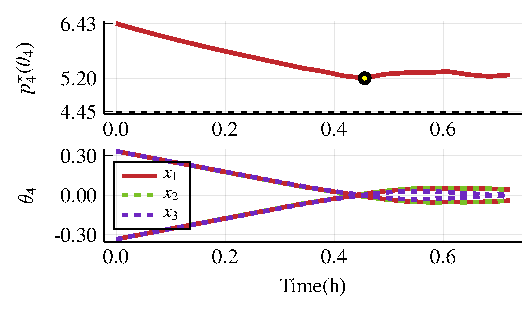}
	\end{flushright}
\caption{
\new{ Brockett integrator: degree 4 polynomials and 6 splits.}  The volume of the ROA approximation is shown in the top plot. The bottom plot shows how the split positions evolved during the optimization process. There were 3 splits for each state variable.
The black-and-yellow dot represents the attained minimum. We see that all the splits were close to the origin at the minimum.}
\label{fig:descent_Bro_d4np_6}
\end{figure}

The Figure \ref{fig:valuefun_bro} shows a 1-dimensional line segment parametrized by \(r\in[0,2]\), connecting the attained solution \(\theta^\star\) (\(r=0\)) to the global estimate \(\theta^\star_{\rm g}\) (\(r=1\)). 
Note that the x-axis has a very small scale, concretely \(||\theta^\star - \theta^\star_{\rm g}|| = 0.02\) while the system is bounded between \(\pm 1\). 
\begin{figure}[!htb]
	\centering
	\begin{flushright}
		\inputnamedtex{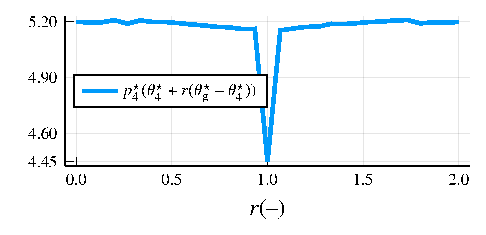}
	\end{flushright}
\caption{\new{Brockett integrator: degree 4 polynomials and 6 splits.} Slice of the value function between the attained optimum (\(r=0\)) and the estimated global optimum (\(r=1\)).
}
\label{fig:valuefun_bro}
\end{figure}

\subsection{High degree}
\label{sec:low_high_order}
This section contains optimization results for the same systems but at a higher degree of the SOS approximation. We do not provide the estimates of global minima, because they would take very long to compute. The high-degree systems had numerical difficulties and their optimization was more demanding than the low-degree case.

We provide two plots for each system, the first one being the application of the same method directly on the high-degree system. The second shall plot objective values of the high-degree system, for the parameter paths obtained from the \emph{low}-degree optimization; the low-degree problems were optimized first and the high-degree system was simply evaluated along their parameter paths. Given the positive results, this technique could be a possible measure to circumvent the inherently bad numerical conditioning of high-degree SOS problems.

\subsubsection*{Double integrator}
Figure \ref{fig:descent_Double_d8np_4} shows the results for the \new{Double integrator with degree 8 polynomials.} We see that the path of the objective values 
is not as smooth as in the low-degree case.

Figure \ref{fig:descent_Double_lowhigh_d8np_4} shows the objective values \new{with degree 8 polynomials while using parameter path obtained with degree 4 and 6 polynomials.}
We see that we can obtain almost the same optimal values much faster (120 times for the degree 4 path and 24 times for the degree 6 path).

\begin{figure}[!htb]
	\centering
	\begin{flushright}
		\inputnamedtex{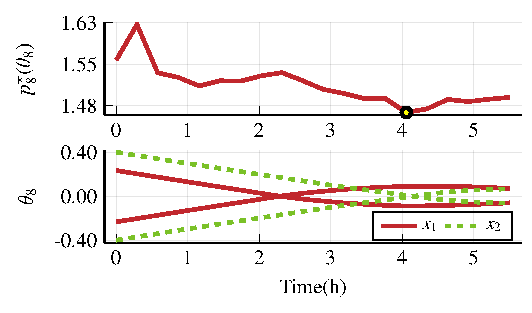}
	\end{flushright}
\caption{\new{Double integrator: degree 8 and 4 splits}. The volume of the ROA approximation is shown in the top plot. The bottom plot shows how the split positions evolved during the optimization process. There were two splits for each state variable.
The black-and-yellow dot represents the attained minimum.}
\label{fig:descent_Double_d8np_4}
\end{figure}

\begin{figure}[!htb]
	\centering
	\begin{flushright}
		\subfloat
		{
		  \inputnamedtex{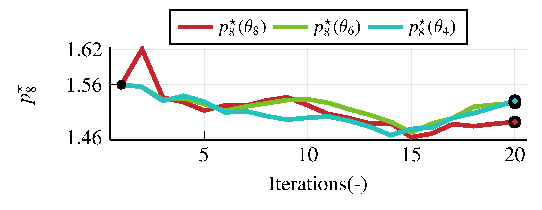}
		}
	\end{flushright}
	\begin{flushright}
		\subfloat
		{
		  \inputnamedtex{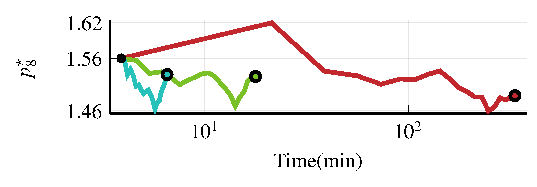}
		}
	\end{flushright}
\caption{
\new{Double integrator: degree 8 polynomials and 4 splits, split parameters computed with lower degree polynomials.}
The computation times per iteration were \(7.91\)s, \(41.6\)s, and \(990.5\)s for degrees 4, 6, and 8 in this order. We can see that all the trajectories reach similar optimal values, while the lower-degree ones were calculated significantly faster. 
}
\label{fig:descent_Double_lowhigh_d8np_4}
\end{figure}





\subsubsection*{Brockett integrator}
Figure \ref{fig:descent_Bro_d6np_6} shows the results for the Brockett integrator \new{with degree 6 polynomials}. Again, we see that the objective path is not 
as smooth as in the low-degree case.

Figure \ref{fig:descent_Bro_lowhigh_d6np_6} shows the results from using the split parameter path obtained with \new{degree 4 polynomials}. We see that in this case, the found minimum was improved by approximately \(55\%\) compared to the initial guess.

\begin{figure}[!htb]
	\centering
	\begin{flushright}
		\inputnamedtex{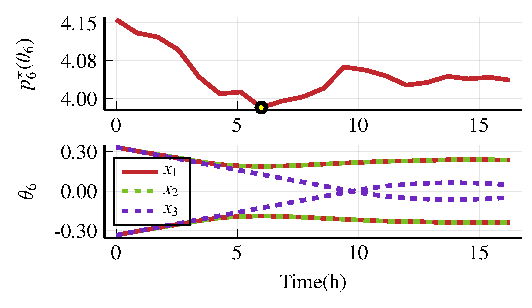}
	\end{flushright}
\caption{ \new{Brockett integrator: degree 6 polynomials and 6 splits}.
The volume of the ROA approximation is shown in the top plot. The bottom plot shows how the split positions evolved during the optimization process. There were 3 splits for each state variable. The black-and-yellow dot represents the attained minimum.}
\label{fig:descent_Bro_d6np_6}
\end{figure}

\begin{figure}[!htb]
	\centering
	\begin{flushright}
		\subfloat
		{
		  \inputnamedtex{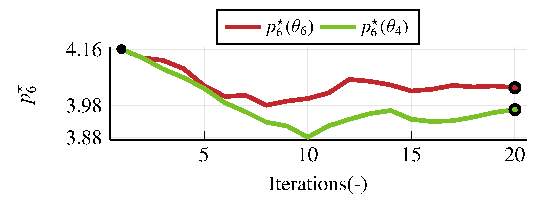}
		}
	\end{flushright}
	\begin{flushright}
		\subfloat
		{
		  \inputnamedtex{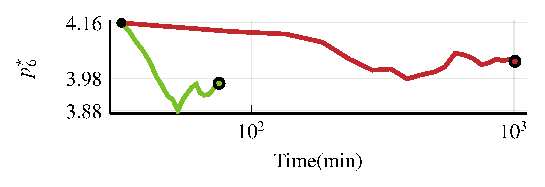}
		}
	\end{flushright}
\caption{
	\new{ Brockett integrator with degree 6 polynomials and 6 splits with split-parameter paths obtained from degree 4 (green) and degree 6 (red) approximations. 	We see that the ROA approximations for parameters optimized with degree 4 polynomials are better than those obtained by optimizing the parameters with degree 6 polynomials. This may be caused by numerical instability present in high-order sum-of-squares approximations. }
}
\label{fig:descent_Bro_lowhigh_d6np_6}
\end{figure}

\addtolength{\textheight}{0cm}   

\section{Conclusion}
\label{sec:conclusion}
We have presented a method for optimizing 
the split locations for ROA computation via conic differentiation \new{of the underlying SDP} problem and first-order methods. We have adopted a differentiation method for multi-parametric SDPs and improved its scaling by a considerable margin as can be seen in Figures \ref{fig:comp_qr_lsqr_fin} and \ref{fig:comp_qr_lsqr_fin_fixed_param}.
In Section \ref{sec:numerical_examples}, we demonstrated the viability of the method by
optimizing the split ROA problem with an off-the-shelf first-order method without any problem-specific tuning. We have managed to improve the objective values by \(60\%\) improvement across the presented examples, where \(100\%\) would mean attaining the (estimated) global optimum.

Finally, we have discussed the possibility of saving time and avoiding numerical issues of high-degree approximations by using optimal solutions from low-degree approximations as starting points.

\new{Possible generalizations of the presented method include using non-axis aligned splits or, more generally, a suitably parametrized semialgebraic partition of the constraint set. Another avenue for future research is merging the proposed approach with sparsity exploiting methods of ~\cite{wang2023exploiting,tacchi2020approximating,schlosser2020sparse}, which are complementary means of reducing the computational complexity of the moment-sum-of-squares approaches for dynamical systems.}

\bmsection*{Acknowledgment}
The authors would like to thank Antonio Bellon for fruitful discussions
about the generic properties of SDPs.
\label{sec:future_work}




\bibliography{library_roa_diff_wiley}

\end{document}